\documentclass[a4paper,12pt,oneside]{amsart}
\usepackage{amsmath}
\usepackage{amssymb}
\usepackage{amsthm}
\usepackage[english]{babel}
\usepackage{esint}
\usepackage{graphicx}
\usepackage{layout}
\usepackage{multido}
\usepackage{xcomment}
\usepackage[usenames,dvipsnames]{color}
\usepackage{comment}
\usepackage{mathrsfs}
\usepackage[final]{pdfpages}
\usepackage{pst-func}
\usepackage{pstricks}
\usepackage{pst-node}
\usepackage{url}

\newcommand{\dist}{\mathrm{dist}}

\newcommand{\nint}{\int_{\RN}}

\newcommand{\no}[1]{\|#1\|}

\newcommand{\supp}{\mathrm{supp}\,}

\newcommand{\R}{\mathbb{R}}

\newcommand{\C}{\mathbb{C}}

\newcommand{\ra}{\rightarrow}

\newcommand{\wk}{\rightharpoonup}
\newtheorem*{theoremA}{Theorem\,A}
\newtheorem*{theoremB}{Theorem\,B}
\ifx\theorem\undefined
\newtheorem{theorem}{Theorem}
\newtheorem*{theorem*}{Theorem}
\fi
\ifx\lemma\undefined
\newtheorem{lemma}{Lemma}[]
\fi

\ifx\prop\undefined
\newtheorem{prop}{Proposition}
\newtheorem*{prop*}{Proposition}
\fi

\theoremstyle{definition}
\newtheorem{defn}{Definition}[]
\theoremstyle{remark}

\newcommand{\var}{\varepsilon}

\newcommand{\om}{\omega}

\newcommand{\vfi}{\varphi}
\newcommand{\V}{V}

\newcommand{\G}{G}
\newcommand{\F}{F}
\newcommand{\N}{n}
\newcommand{\RN}{\mathbb{R}\sp\N}
\newcommand{\re}{\mathrm{Re}}
\newcommand{\cg}{\beta}
\newcommand{\Ce}{C}
\newcommand{\CH}{\mathbf{C}}
\newcommand{\HH}{H}
\newcommand{\M}{M_\Ce}
\newcommand{\D}{D}

\newcommand{\EH}{\mathbf{E}}

\newcommand{\VH}{\mathbf{V}}
\newcommand{\J}{J}
\newcommand{\kk}{j}
\newcommand{\ii}{i}
\newcommand{\pt}{\D}

\newcommand{\Y}{X}
\newcommand{\cz}{c_0}
\newcommand{\co}{c_1}

\newcommand{\m}{m_2\sp 2}
\newcommand{\f}{I}
\renewcommand{\dist}{d}
\usepackage[colorlinks=true,linkcolor={black},citecolor={black}]{hyperref}
\begin{document}
\title[Orbitally stable coupled standing-waves]%
{On the orbital stability of standing-waves solutions
to a coupled non-linear Klein-Gordon equation}
\author{Garrisi Daniele}
\address[Garrisi Daniele]%
{Math Sci. Bldg Room \# 302, POSTECH, Hyoja-Dong, Nam-Gu, Pohang, 
Gyeongbuk, 790-784, Republic of Korea}
\email{garrisi@postech.ac.kr}
\thanks{This work was supported by Priority Research Centers Program 
through the National Research Foundation of Korea (NRF) funded by the 
Ministry of Education, Science and Technology (Grant \#2010-0029638).}
\date{\today}
\begin{abstract}
We consider a system of two coupled non-linear Klein-Gordon equations.
We show the existence of standing waves solutions and the existence
of a Lyapunov function for the ground state.
\end{abstract}
\maketitle
\section*{Introduction}
\noindent 
The purpose of this work is to set some basic results to prove the
orbital stability of standing-waves solutions to a coupled 
non-linear Klein-Gordon equation
\begin{equation*}
\label{eq:nlkg0}
\tag{NLKG}
\left\{
\begin{array}{l}
\pt_{tt}\phi_1 -\Delta\phi_1 + m_1 \sp 2\phi_1 + \pt_1 \F(\phi) = 0\\
\pt_{tt}\phi_2 -\Delta\phi_2 + \m\phi_2 + \pt_2 \F(\phi) = 0.
\end{array}
\right.
\end{equation*}
The existence of standing-waves is obtained through a variational approach
which provides with a solution $ (u,\om) $ of the elliptic system
\begin{equation*}
\label{eq:ES}
\tag{ES}
\left\{
\begin{array}{l}
-\Delta u_1 + m_1 \sp 2 u_1 + \D_1 \F (u) = \om_1 \sp 2 u_1\\
-\Delta u_2 + \m u_2 + \D_2 \F (u) = \om_2 \sp 2 u_2.
\end{array}
\right.
\end{equation*}
For scalar field equations, the development of tools for rigorous
proofs of the orbital stability of standing-waves for the nonlinear
Klein-Gordon equation, or the Schr\"odinger equation is relatively recent.
The first results for NLKG are due to J.~Shatah in \cite{Sha83} 
(and generalised to coupled NLKG in \cite{ZGG10})
where it is shown that solutions of
\[
\tag{0.3$ \om $}
- \Delta u + (1 - \om\sp 2) u + \F'(u) = 0
\]
which are minimizers of some functional $ \J_\om $ on a natural
constraint $ M_\om $, are stable for the values of $ \om $ where
the function $ \om\mapsto \inf_{M_\om} \J_\om $ is convex; for the NLS 
(and other scalar field equations), in the work of T.~Cazenave and 
P.~L.~Lions \cite{CL82} it is proved the orbital stability of solutions
obtained as minimizers of the energy functional $ \mathcal{E} $, 
\cite[p.\,3]{CL82}, on the 
constraint 
\[
N_\lambda := \{\no{u}_{L\sp 2 (\RN)} \sp 2 = \lambda\}.
\]
This result has been generalised to a large class of non-linearities
for NLS in \cite{BBGM07} and in \cite{MMP10} for some class of coupled
NLS. In \cite{BBBM10} the orbital stability of NLKG for a class of 
solutions (in general different from \cite{Sha83}) obtained as
minimizers of the functional 
\[
E(u,\om) = \frac{1}{2}\nint \Big(|\D u|\sp 2 + m\sp 2 u\sp 2\Big) 
+ \nint \F(u)
+ \frac{\om\sp 2}{2} \nint u\sp 2
\]
on the costraint 
\[
\M := \big\{(u,\om)\,|\,\om \no{u}_{L\sp 2 (\RN)} \sp 2 = \Ce\big\}
\]
is considered. 
The hypotheses on the non-linearity $ \F $ are very general and the
stability is proved under the assumption that local solutions of NLKG
exist in $ H\sp 1 $ and radially symmetric minimizers are isolated 
in $ \M $.\vskip .3em

\noindent This work deals with standing-waves solutions obtained as minimizer
of the energy functional
\[
\label{eq:E}
\tag{E}
E(u,\om) = \frac{1}{2}\sum_{\kk = 1} \sp 2
\nint \Big(|\D u_\kk|\sp 2 + m_\kk\sp 2 u_\kk\sp 2 + 
\om_\kk \sp 2 u_\kk\sp 2\Big) 
+ \nint \F(u)
\]
on the constraint
\[
\label{eq:M}
\tag{C}
\M := \big\{(u,\om)\,|\,
\om_\kk \no{u_\kk}_{L\sp 2 (\RN)} \sp 2 = \Ce_\kk\big\}.
\]
The utility of this variational setting is two-fold: firstly, the 
Euler-Lagrange equations correspond to a solution of \eqref{eq:ES},
thus we do not need a further discussion on the sign of the 
Lagrange multipliers. Moreover, due to the symmetry of
the Lagrangian of \eqref{eq:nlkg0}, for a smooth solution $ \phi $ we have
the conservation laws
\begin{gather*}
\label{eq:EN}
\tag{$ \mathbf{E} $}
\EH(\phi,\phi_t) = \frac{1}{2} \nint |\phi_t|\sp 2 + 
|\D\phi|\sp 2 
+ 2V(\phi)\\
\label{eq:C}
\tag{$ \mathbf{C} $}
\CH_\kk (\phi,\phi_t) = -\text{Im}\nint \phi_t \sp\kk 
\overline{\phi}_\kk (t,\cdot),
\ 1\leq\kk\leq 2.
\end{gather*}
These correspond to $ E $ and $ \Ce_\kk $ on standing-waves solutions.
The main theorems are the following:
\begin{theoremA}
\hypertarget{thm:A}{}
Given a minimizing sequence $ (u_n,\om_n) $ for $ E $ over $ \M $,
there exists a minimizer $ (u,\om) $ and 
$ (y_n)_{n\geq 1}\subset\RN $ such that, up to extract a subsequence
\[
u_n \sp\kk = u_\kk(\cdot + y_n) + o(1)\text{ in } H\sp 1 (\RN),
\quad\om_n\ra\om\text{ in } \R\sp 2
\]
for $ 1\leq\kk\leq 2 $.
\end{theoremA}
As in the scalar case in \cite{BBBM10}, compactness 
of minimizing sequences is proved for the simpler functional 
\[
\J(u) = \frac{1}{2}\sum_{\kk = 1} \sp 2\nint |\D u_\kk|\sp 2
+ \nint\F(u).
\]
on the constraint
\[
N_\rho := \{u\,|\,\no{u_\kk}_{L\sp 2 (\RN)} \sp 2 = \rho\}.
\]
In the scalar case, the compactness of the minimizing sequences of 
$ \J $, \cite{BBGM07}, is achieved by proving the sub-additivity property
of $ I_{\rho} := \inf_{N_\rho} J $, that is
\[
I_\rho < I_\tau + I_{\rho - \tau},\quad 0 < \tau < \rho.
\]
We follow the same approach. However, while in the scalar case, 
such inequality can be proved by rescaling two minimizing sequences in 
$ N_\tau $ and $ B_{\rho - \tau} $, a more effort is needed
for systems; we address this property to Section~\ref{sect:subadd}.
In Lemma~\ref{lem:subadd} we show that there exists $ D > 0 $, depending
only on $ \rho $ and $ \tau $ such that
\[
I_\rho < I_\tau + I_{\rho - \tau} - D.
\]
The inequality uses the symmetric decreasing rearrangement,
\cite{Kaw85}. The idea we follow is that, if two bumps 
$ u\in N_\tau $ and $ v\in N_{\rho - \tau} $ have small interaction,
then
\[
\no{\D w\sp *}\sp 2 < \no{\D u}\sp 2 + \no{\D v}\sp 2 - D
\]
where $ w = u + v $ and $ w\sp * $ is the symmetric rearrangement.
The second theorem concerns the properties of two subsets of
the phase space of \eqref{eq:nlkg0}, 
$ H\sp 1 (\RN,\C\sp 2)\oplus L\sp 2 (\RN,\C\sp 2) $.
To a minimizer $ (u,\om) $ of $ E $ over $ M_\Ce $ we can associate
\[
\Gamma(u,\om) =
\left\{
\begin{array}{c}
\left(\lambda u(\cdot + y),-i\om\lambda u (\cdot + y)\right)\\\\
(\lambda_1,\lambda_2,y)\in\C\times\C\times\RN,\quad 
|\lambda_1 | = |\lambda_2 | = 1
\end{array}
\right.,
\]
and to a constraint $ M_\Ce $, we can associate 
\begin{gather*}
\tag{GS}
\label{eq:GS}
\Gamma_\Ce = \bigcup\big\{\Gamma(u,\om)\,|\,(u,\om)\in K_\Ce\big\},
\end{gather*}
called \textsl{ground state}, where
\[
m_\Ce := \inf_{\M} E,\quad 
K_\Ce := \left\{(u,\om)\,|\,E(u,\om) = m_\Ce\right\}.
\]
\begin{theoremB}
\hypertarget{thm:B}{}
Given a sequence 
\[
(\Phi_n)_{n\geq 1}\subset 
H\sp 1 (\RN,\C\sp 2)\oplus L\sp 2 (\RN,\C\sp 2),
\]
then $ \dist(\Phi_n,\Gamma_{\Ce})\ra 0 $ if and only if
\[
\mathbf{E}(\Phi_n)\ra m_{\Ce},\quad
\mathbf{\Ce}_\kk (\Phi_n)\ra \Ce_\kk.
\]
for  $ 1\leq \kk\leq 2 $.
\end{theoremB}
A proof of this theorem in the scalar case can be found in
\cite{BBBM10} under the assumption that the NLKG is locally well-posed.
In our proof we drop this assumption. The keypoint of the proof
lies in the following property:
given $ \phi\in H\sp 1 (\RN,\C) $ such that $ |\phi| > 0 $ everywhere and
\[
\no{\D\phi} = \no{\D|\phi|}
\]
there exists $ \lambda\in\C $ such that $ |\lambda| = 1 $ and
$ \phi = \lambda|\phi| $.
We show this in Lemma~\ref{lem:inequalities} for
$ H\sp 1 (\RN,\R\sp m) $ and $ m\geq 1 $. A similar
property is shown in \cite[Theorem~7.8]{LL01} under the stronger
assumption that $ |\phi_\kk| > 0 $ 
for some $ 1\leq\kk\leq m $.\vskip .3em
The non-linear term $ \F $ is assumed to be continuously differentiable 
with subcritical growth and can be written as non-negative
perturbation of a coupling term
\[
\F(u) = -\cg |u_1 u_2|\sp\gamma + \G,\quad\G\geq 0.
\]
Theorems~\hyperlink{thm:A}{A} and \hyperlink{thm:B}{B} are addressed
to the proof that $ \Gamma_\Ce $ is a stable subset of the phase
space
\[
\Y := H\sp 1 (\RN,\C\sp 2)\oplus L\sp 2 (\RN,\C\sp 2).
\]
Thus, it is very natural to ask whether we have local existence of
solutions to \eqref{eq:nlkg0} with initial data in $ \Y $.
However, from known results on the non-linear scalar wave equation,
we can expect local existence only
\[
\Y_k := H\sp k (\RN,\C\sp 2)\oplus H\sp{k - 1} (\RN,\C\sp 2),\quad k > n/2
\]
with the general assumptions we make on $ \F $. Moreover, for $ k = 1 $ 
even conservation laws \eqref{eq:EN} and 
\eqref{eq:C} are not known to hold for every non-linearity. 
In order to obtain the stability of $ \Gamma(u,\om) $ it seems that
the non-degeneracy condition
\[
\{(u(\cdot + y),\om)\,|\,y\in\RN\}\text{ is isolated in } K_\Ce
\]
is rather necessary. We do not tackle in this work the problem of the 
existence of local solutions and the non-degeneracy condition.

Numerical results on the existence of standing-waves have been
obtained in \cite{BH08} when $ \N = 3 $ and critical exponents.

\textbf{Acknowledgements.}
I would like to thank professor Vieri Benci and professor 
Jaeyoung Byeon for their aid and helpful suggestions, as well as
professor Claudio Bonanno and Jacopo Bellazzini.
\section{Regularity properties}
\noindent We fix $ \N\geq 3 $ and recall the well-known inequalities
for a function $ u\in H\sp 1 (\RN) $
\begin{gather}
\label{eq:sob-1}
\no{u}_{L\sp {2\sp *}}\leq S\no{\D u}_{L\sp 2}\\
\label{eq:sob-2}
\no{u}_{L\sp p}\leq \no{u}_{L\sp 2}
+ S\no{\D u}_{L\sp 2}
\end{gather}
for some $ S > 0 $ and for every $ 2\leq p\leq 2\sp * $
and $ 2\sp * = 2\N/\N - 2 $, check \cite[Corollaire~IX.10,p.\,165]{Bre83}.
Given an integer $ m $, we set
\[
\HH := \bigoplus_{k = 1} \sp m H\sp 1 (\RN).
\]
Given $ u\in L\sp p (\RN,\R\sp m) $, we also set
\[
\no{u}_p = \no{u}_{L\sp p},\quad \no{u} = \no{u}_{L\sp 2}.
\]
On $ \HH $ we consider the norm defined as 
\[
\no{u}_{\HH} \sp 2 := \sum_{\kk = 1} \sp m \no{u_\kk}\sp 2 + 
\no{\D u_\kk}\sp 2.
\]
\begin{defn}
A real-valued function $ \F\colon \RN\ra\R\sp m $ is a
\textsl{combined power-type} if there exists a constant $ c > 0 $ and
$ p\leq q $ such that 
\[
|\F(u)|\leq c(|u|\sp p + |u|\sp q)
\]
for every  $ u\in\RN $. If $ p = q $, we say that $ F $ is a
\textsl{power-type}.
\end{defn}
\begin{prop}
\label{prop:cpt}
Let $ \F\colon\RN\ra\R\sp m $ be a differentiable function such
that $ \F(0) = 0 $ and there are $ p\leq q $ such that
\begin{equation}
\label{eq:prop:cpt-1}
|\D \F(u)|\leq \co (|u|\sp{p - 1} + |u|\sp{q - 1}).
\end{equation}
Then, there are $ F_p,F_q\colon\RN\ra\R\sp m $ differentiable such that
$ F_p (0) = F_q (0) = 0 $ and 
\begin{gather*}
\F = \F_p + \F_q\\
|\D \F_p (u)|\leq c_{p - 1} |u|\sp{p - 1},\quad |\D \F_q (u)|\leq c_{q - 1}
|u|\sp{q - 1}
\end{gather*}
\end{prop}
\begin{proof}
Because $ \F(0) = 0 $, and by the Fundamental Theorem of the Calculus, 
it follows from the hypotheses that 
\begin{equation}
\label{eq:combinedpt}
|\F(u)|\leq\cz (|u|\sp{p} + |u|\sp{q}) 
\end{equation}
in fact $ \cz $ could be chosen to be $ \sqrt{m}\co/p $. 
Let $ \eta\in C\sp 1 (\R\sp m,\R) $ be a non-negative function such that
\[
\eta(u) =
\begin{cases}
1 & \text{ if } |u|\leq 1\\
0 & \text{ if } |u|\geq 2
\end{cases}
\]
with $ \eta\leq 1 $ and $ |\D\eta|\leq 2 $. 
On $ B(0,2) $, by \eqref{eq:prop:cpt-1}, we have 
\[
\begin{split}
|\eta \D F(u)|\leq& \co (|u|\sp{p - 1} + |u|\sp{q - 1})
= \co\big[|u|\sp{p - 1} + 2\sp {q - 1} (|u|/2)\sp{q - 1}\big]\\
\leq & \co(|u|\sp{p - 1} + 2\sp{q - p} |u|\sp{p - 1}) 
= \co (1 + 2\sp{q - p}) |u|\sp{p - 1}.
\end{split}
\]
The second inequality follows from $ p\leq q $. 
Because $ \eta $ vanishes outside $ B(0,2) $ the inequality above holds in 
$ \R\sp m $. On the anulus $ C(1,2) $ we have
\[
\begin{split}
|F\D\eta|\leq& 2\cz (|u|\sp{p - 1} + |u|\sp{q - 1})\leq 
2\cz (1 + 2\sp{q - p - 2}) |u|\sp{p - 1}.
\end{split}
\]
Since $ \D\eta $ vanishes outside $ C(1,2) $, the inequality above
holds in $ \R\sp m $. Combining the last two inequalities, we prove that
$ \D (\eta F) $ is power-type. Similarly, one shows that
\[
|\D(1 - \eta) F|\leq 2(\co + 2\cz) |u|\sp{q - 1}.
\]
We set $ F_p := \eta F $ and $ F_q := (1 - \eta) F $. Thus,
\begin{equation}
\label{eq:pt}
\F = \F_p + \F_q
\end{equation}
is the desired decomposition.
\end{proof}
Let $ \F $ be a real-valued continuously differentiable function on 
$ \R\sp m $ such that
\begin{equation}
\label{eq:combinedptd}
|\D F(u)|\leq \co
(|u|\sp{p - 1} + |u|\sp{q - 1}),\quad \F(0) = 0,\quad
2\leq p\leq q\leq 2\sp *
\end{equation}
for every $ u\in\R\sp m $. For every 
$ u\in\HH $ we have $ F(u)\in L\sp 1 (\RN) $ from inequalities
\eqref{eq:sob-2} and \eqref{eq:combinedpt}.
Therefore we have a well-defined functional
\begin{gather*}
\J\colon \HH\ra\R\\
\J(u) = \frac{1}{2}\nint |\D u|\sp 2
+ \nint\F(u).
\end{gather*}
\begin{prop}
\label{prop:smooth}
The functional $ J $ defined above satisfies the following:
\begin{itemize}
\item[(a)] $ J $ is of class $ C\sp 1 (\HH,\R) $,
\item[(b)] if $ q < 2\sp * $, given a weakly converging sequence 
$ u_n\rightharpoonup u $ in $ \HH $, up to extract a subsequence, we
have
\begin{gather*}
J(u_n - u) = J(u_n) - J(u) + o(1).
\end{gather*}
\end{itemize}
\end{prop}
The proof of (a) uses the same techinque of \cite[Theorem~2.2]{AP93} and
\cite[Theorem~2.6]{AP93} which deals with bounded domains. In fact,
such restriction is not necessary; the proof of (b) is the same as
the scalar case of \cite[Appendix]{BBGM07}.
\begin{proof}
(a).~Since the map $ u\mapsto\no{\D u}_2 \sp 2 $ is smooth on $ \HH $,
we only need to prove that
\[
\mathscr{\F}(u) := \nint \F(u)
\]
is $ C\sp 1 (\HH) $. Moreover, by Proposition~\ref{prop:cpt}, we
can suppose that $ \F $ and $ \D F $ are power type non-linearities
and
\begin{equation}
\label{eq:cpt}
|\D \F (u)|\leq \co |u|\sp{p - 1},\quad |\F(u)|\leq \cz |u|\sp p.
\end{equation}
From the first of the two inequalities above, $ |\D\F(u)| $ is
in $ L\sp{p'} $, where $ p' = p/(p - 1) $. Because 
$ |\D_\kk \F (u)|\leq |\D \F (u)| $ 
the application
\begin{equation}
\mathscr{G}_\kk\colon H\sp 1\ra L\sp{p'},\quad u\mapsto \D_\kk \F (u).
\end{equation}
is well defined for every $ 1\leq \kk\leq m $.
We prove that it is also continuous. To this end,
let $ u_n\ra u $ be a converging sequence in $ \HH $; we show
that $ \mathscr{G}_\kk (u_n) $ has a converging subsequence and that
all the converging subsequences have the same limit 
$ \mathscr{G}_\kk (u) $. Thus,
\[
\mathscr{G}_\kk (u_n)\ra\mathscr{G}_\kk (u).
\]
Up to extract a subsequence, we can suppose that there exists 
$ v\in L\sp p (\RN) $ such that
\[
u_n\ra u,\quad |u_n \sp\kk|\leq v
\]
almost everywhere, for $ 1\leq\kk\leq m $.
Because $ \D_\kk\F $ is continuous, by the convergence above, we have
\begin{gather*}
\D_\kk \F (u_n)\ra \D_\kk \F (u)\text{ pointwise a.e.}\\
|\D_\kk \F (u_n) - \D_\kk \F (u)|\sp{p'}\leq (2\co)\sp{p'} 
|v|\sp p\in L\sp 1 (\RN).
\end{gather*}
Thus, by the dominate convergence theorem, we obtain the convergence
of $ \mathscr{G}_\kk (u_n) $ to $ \mathscr{G}(u) $.
Now, for every $ u\in\HH $, we consider the linear functional
\begin{gather}
L_u\colon \HH\ra\R\\
\label{eq:differential}
L_u (\vfi) := \sum_{\kk = 1} \sp m \nint \D_\kk \F (u) \vfi_\kk,
\end{gather}
which is well-defined and bounded by the H\"older inequality. 
Next, we show that
\[
\mathscr{\F} (u + \vfi) - \mathscr{\F} (u) - L_u (\vfi) = o(\vfi).
\]
We prove the convergence above on sequences $ \vfi_n\ra 0 $. 
The left term equals
\[
\begin{split}
&\nint\F (u + \vfi_n) - \F (u) - L_u (\vfi_n) \\
=& \int_0 \sp 1 \nint\langle\D\F (u + t\vfi_n),\vfi\rangle - L_u (\vfi_n) \\
=&\sum_{\kk = 1} \sp m\int_0 \sp 1 \nint 
\left(\D_\kk \F (u + t\vfi_n) - \D_\kk \F (u)\right)\vfi_n \sp \kk.
\end{split}
\]
Thus, by the definition of $ \mathscr{G}_\kk $ and the H\"older inequality,
we obtain
\[
\begin{split}
&\left|\nint\F (u + \vfi_n) - \F (u) - L_u (\vfi_n)\right| \\
\leq&\sum_{\kk = 1} \sp m \no{\vfi_n \sp \kk}_p 
\int_0 \sp 1 \no{\mathscr{G}_\kk (u + t\vfi_n) - 
\mathscr{G}_\kk (u)}_{p'}\\
\leq&\sqrt{1\vee S}\,\no{\vfi_n}_{\HH} \left[\sum_{\kk = 1} \sp m
\left(\int_0 \sp 1 \no{\mathscr{G}_\kk (u + t\vfi_n) - 
\mathscr{G}_\kk (u)}_{p'}\right)\sp 2 \right]\sp{1/2}
\end{split}
\]
where the second inequality still follows from \eqref{eq:sob-2}
and the Schwarz inequality for the eucliden product on $ \R\sp m $.
By the continuity of $ \mathscr{G}_\kk $, the functions
\[
g_n \sp\kk (t) := \no{\mathscr{G}_\kk (u + t\vfi_n) - 
\mathscr{G}_\kk (u)}_{p'}
\]
are continuous on the unit interval and converge pointwise. Because
the sequence $ \vfi_n $ is bounded in $ \HH $, they
are also uniformly bounded from above. Then by the dominated convergence 
theorem, we have
\[
\int_0 \sp 1 g_n \sp \kk (t) \ra 0.
\]
Thus, the last term of the previous inequality is $ o(1) \no{\vfi}_\HH $
which proves that $ \mathscr{\F} $ is continuous and differentiable in $ u $
and
\[
\D\mathscr{\F} (u) = L_u.
\]
Finally, we observe that by the continuity of $ \mathscr{G}_\kk $ and
the definition of $ L_u $ in \eqref{eq:differential}, the map
\[
\D\mathscr{\F}\colon \HH\ra\mathcal{L}(\HH,\HH\sp*)
\]
is continuous. Thus, $ \mathscr{\F}\in C\sp 1 (\HH,\R) $. 
Then, $ \J\in C\sp 1 (\HH,\R) $.
\vskip .1em
\noindent(b).~For $ u\mapsto\no{\D u}\sp 2 $ the property follows easily from
\[
\begin{split}
\no{\D(u_n - u)}\sp 2 =& \no{\D u_n}\sp 2 + \no{\D u}\sp 2 - 2(\D u_n,\D u)\\
=& \no{\D u_n}\sp 2 - \no{\D u}\sp 2 + o(1).
\end{split}
\]
Again, from Proposition~\ref{prop:cpt} we can suppose that \eqref{eq:cpt}
holds. We set $ v_n := u_n - u $. 
Let us fix $ \var > 0 $. We prove that there exists a subsequence
of $ (v_n) $ such that
\[
\lim_{n\ra +\infty}
|\mathscr{\F} (u + v_n) - \mathscr{\F} (u) - \mathscr{\F} (v_n)| < \var.
\]
Given $ R > 0 $, we have
\[
\begin{split}
&\mathscr{\F} (u + v_n) - \mathscr{\F} (u) - \mathscr{\F} (v_n) \\
=&\nint \F (u + v_n) - \F (u) - \nint \F (v_n) \\
=&\int_{B_R} \F (u + v_n) - \F (u) - \int_{B_R} \F (v_n) \\
+&\int_{B_R \sp c} \F (u + v_n) - \F (v_n) - \int_{B_R \sp c} \F (u)
=: A + B
\end{split}
\]
We estimate separately the summands of the last term of the equality.
Since $ \F $ differentiable, we have
\[
\begin{split}
B = &\int_{B_R \sp c} \F (u + v_n) - \F (v_n) - \int_{B_R \sp c} \F (u)\\
=&\int_{B_R \sp c} \int_0 \sp 1 \langle\D\F (v_n + tu),u\rangle - 
\int_{B_R \sp c} \F (u) =: B_1 + B_2.
\end{split}
\]
By \eqref{eq:cpt} and 
the H\"older inequality, we have
\begin{gather*}
  |B_1| \leq \co\int_{B_R \sp c} |u + tv_n|\sp{p - 1} |u|
\leq \co\sup_{n,t} \no{u + tv_n}_{L\sp p (B_R \sp c)} \sp{p - 1}
\,\no{u}_{L\sp p (B_R \sp c)}\\
|B_2|\leq \cz \no{u}_{L\sp p (B_R \sp c)} \sp p.
\end{gather*}
Because $ v_n $ converges weakly, the supremum above is finite.
Since $ u\in L\sp p (\RN) $, there exists $ R(\var) $ such that 
$ |B_i|\leq\var/4 $. Similarly, we have
\[
\begin{split}
A = 
&\int_{B_{R(\var)}} \F (u + v_n) - \F (u) - 
\int_{B_{R(\var)}} \F (v_n) \\
=& \int_{B_{R(\var)}} \int_0 \sp 1 \langle\D\F (u + tv_n),v_n\rangle -
\int_{B_{R(\var)}} \F (v_n) =: A_1 + A_2.
\end{split}
\]
From \eqref{eq:cpt}, we have
\begin{gather*}
|A_1|\leq\co\int_{B_{R(\var)}} |u + tv_n|\sp{p - 1} |v_n|\leq
\co \sup_{n,t} \no{u + tv_n}_{L\sp p (B_{R(\var)})} \sp{p - 1}
\,\no{v_n}_{L\sp p (B_{R(\var)})}\\
|A_2|\leq\cz\no{v_n}_{L\sp p (B_{R(\var)})} \sp p.
\end{gather*}
Because $ p < 2\sp * $ the inclusion of
$ L\sp p (B_{R(\var)}) $ in $ H\sp 1 (B_{R(\var)}) $ is compact,
\cite[Th\'eor\`eme~IX.16,p.\,169]{Bre83}. Thus, we
can extract a subsequence such that $ v_n \ra 0 $ in $ L\sp p (B_{R(\var)}) $. 
If we choose $ n $ large enough, we obtain $ |A_i|\leq\var/4 $. 
\vskip .4em
\noindent
If we repeat the same argument for $ \var_k = 1/k $ we obtain
subsequences 
\[
(v_{n,k})\subseteq\dots\subseteq (v_{n,2})\subseteq (v_{n,1}).
\]
Let $ n_k $ be such that
\[
|\mathscr{\F} (u + v_{n_k,k}) - \mathscr{\F} (u) - \mathscr{\F} (v_{n_k,k})| 
< \frac{1}{k}.
\]
Thus $ w_k := v_{n_k,k} $ is a subsequence of $ (v_n) $. Then, the
sequence $ u_{n_k} := w_k + u $ is a subsequence of $ (u_n) $ and
satisfies the required properties.
\end{proof}
\section{The variational setting}
\noindent 
Throught this and the next sections we assume $ m = 2 $ and the following
properties on $ \F $:
\begin{align}
\label{eq:A1}
\tag{A1}
&\F(u) = - \cg |u_1 u_2|\sp\gamma + \G(u),\quad 1 < \gamma < 1 + 2/\N,\\
\label{eq:A2}
\tag{A2}
&|\D \G(u)| \leq \co(|u|\sp{p - 1} + |u|\sp{q - 1}),
\quad 2\gamma < p\leq q < 2\sp *,\\
\label{eq:A3}
\tag{A3}
&\G(u) = \G(|u_1|,|u_2|),\quad G\geq 0.
\end{align}
Finally, we assume that $ \G $ is well-behaved with respect to the
Steiner rearrangement. That is, given $ u_1,u_2\in H\sp 1 (\RN) $,
and denoting by $ u_1 \sp * $ and $ u_2 \sp * $ their Steiner 
symmetric rearrangements (check \cite{Kaw85}), we have
\begin{equation}
\label{eq:A4}
\tag{A4}
\nint G(u_1 \sp *,u_2 \sp *)\leq\nint G(u_1,u_2).
\end{equation}
\renewcommand{\theequation}{\arabic{equation}}
The assumptions (\ref{eq:A1},\ref{eq:A2}) are the natural 
extension of the hypothesis ($ F_p,F_0 $) and ($ F_2 $) made in the scalar 
case by V.~Benci,~M.~Ghimenti 
\textsl{et al.},~\cite{BBGM07}. The proof of the next Lemma, which we
include for the sake of completeness, is similar to the ones of
\cite[Lemma~5,Proposition~7]{BBGM07}.
\begin{lemma}
\label{lem:general}
For every $ \rho\in\R\sp 2 $ with $ \rho_\kk > 0 $, we have 
\begin{enumerate}
\item $ \inf_{N_\rho} \J =: \f_\rho $ is finite and negative,
\item minimizing sequences of $ \J $ on $ N_{\rho} $ are bounded,
\end{enumerate}
\end{lemma}
\begin{proof}
(i).~Let $ v\in H\sp 1 $ be such that $ \no{\D v} = \no{v} = 1 $.
Given $ s_i,R > 0 $, we define
\[
u(x) = (s_1 v(x/R),s_2 v(x/R)).
\]
By a change of variable, it can be easily checked that
\begin{equation*}
\no{u_\kk} \sp 2 = s_\kk \sp 2 R\sp\N,\quad
\no{\D u_\kk} \sp 2 = s_\kk \sp 2 R\sp{\N - 2}.
\end{equation*}
We choose $ R $ and $ s $ such that $ \rho_\kk = s_\kk \sp 2 R \sp \N $.
Thus, $ s_2 = \lambda s_1 $, where $ \lambda\sqrt{\rho_1} := \sqrt{\rho_2} $.
We have
\[
\no{\D u}\sp 2 = \frac{1}{2}
\left(s_1 \sp 2 R\sp{\N - 2} + s_2 \sp 2 R \sp{\N - 2}\right)
= \frac{R\sp{-2} (\rho_1 + \rho_2)}{2}
\]
and
\begin{gather*}
\begin{split}
\nint |u_1 u_2|\sp\gamma =& R\sp n (s_1 s_2)\sp{\gamma}
\no{v}_{2\gamma} \sp{2\gamma} = R\sp n \lambda\sp\gamma s_1 \sp{2\gamma}
\no{v}_{2\gamma} \sp{2\gamma} \\
=& R\sp{n\left(1 - \gamma\right)} (\lambda\rho_1)\sp{\gamma}
\no{v}_{2\gamma} \sp{2\gamma}
\end{split}\\
\begin{split}
\nint \G(u)&\leq \cz\nint |u|\sp p + |u|\sp q = 
\cz R\sp n\left(s_1 \sp p (1 + \lambda\sp 2)\sp{p/2} + s_1 \sp q
(1 + \lambda\sp 2)\sp{q/2}\right)\\
&=\cz R\sp{n\left(1 - \frac{p}{2}\right)} (1 + \lambda\sp 2)\sp{p/2}
+ \cz R\sp{n\left(1 - \frac{q}{2}\right)} (1 + \lambda\sp 2)\sp{q/2}.
\end{split}
\end{gather*}
The constant $ \cz $ above follows from \eqref{eq:A2}. From
the definition of $ \J $ and \eqref{eq:A1} and \eqref{eq:A2}, there exists a 
positive constant $ c > 0 $ such that
\[
\J(u)\leq c \left(R\sp{-2} - R\sp{n(1 - \gamma)} + 
R\sp{n\left(1 - \frac{p}{2}\right)} + R\sp{n\left(1 - \frac{q}{2}\right)}
\right).
\]
For $ R $ large enough the right term of the inequality above is negative. 
Now we prove 
that the infimum of $ \J $ is finite. We fix $ \rho $ as above. 
By the H\"older inequality and \eqref{eq:A3}, we have
\begin{equation}
\label{eq:young}
\begin{split}
J(u)\geq&\frac{1}{2}\sum_{\kk = 1}\sp 2 \no{\D u_\kk} \sp 2 - 
2\cg\left(\no{u_1}_{2\gamma}\no{u_2}_{2\gamma}\right) \sp\gamma\\
\geq&\frac{1}{2}\sum_{\kk = 1}\sp 2 \left(\no{\D u_\kk} \sp 2 - 
2\cg\no{u_\kk}_{2\gamma} \sp{2\gamma}\right)
\end{split}
\end{equation}
From \eqref{eq:sob-1} and the interpolation inequality, we have
\begin{equation}
\label{eq:sob-3}
\begin{split}
\no{u_\kk}_{2\gamma} &\leq
\no{u_\kk} \sp{1 - \frac{\N}{2} + \frac{\N}{2\gamma}} 
\no{u_\kk}_{2\sp*} \sp{\frac{\N}{2} - \frac{\N}{2\gamma}}\\
&\leq S\sp{\frac{\N}{2} - \frac{\N}{2\gamma}}
\no{u_\kk} \sp{1 - \frac{\N}{2} + \frac{\N}{2\gamma}} 
\no{\D u_\kk} \sp{\frac{\N}{2} - \frac{\N}{2\gamma}}\\
&= S\sp{\frac{\N}{2} - \frac{\N}{2\gamma}} \rho_\kk 
\sp{\frac{1}{2}(1 - \frac{\N}{2} + \frac{\N}{2\gamma})} 
\no{\D u_\kk} \sp{\frac{\N}{2} - \frac{\N}{2\gamma}}.
\end{split}
\end{equation}
Then, there exists a constant $ c = c(\beta,\rho_\kk,S) > 0 $ such that
\begin{equation}
\label{eq:coercive}
\begin{split}
J(u)&\geq c\sum_{\kk = 1}\sp 2 \no{\D u_\kk}\sp 2 - 
\no{\D u_\kk} \sp{\frac{\N}{2} - \frac{\N}{2\gamma}}\\
&=c\sum_{\kk = 1} \sp 2 X_\kk \sp 2 - X_\kk 
\sp{\frac{\N}{2} - \frac{\N}{2\gamma}} 
=: g(X)
\end{split}
\end{equation}
where in the last equality $ \no{\D u_\kk}_2 $ has been replaced by $ X_\kk $.
By the hypotheses on $ \gamma $ in \eqref{eq:A1}, $ g $ is bounded from 
below. Thus, $ \f_\rho $ is well-defined and
negative.\vskip .2em
\noindent(ii).~Let $ (u_n) $ be a minimizing sequence on $ N_\rho $.
By definition, $ \no{u_n \sp \kk}_2 \sp 2 = \rho_\kk $ thus constant and 
bounded. Moreover, if a subsequence of $ \no{\D u_n \sp\kk}_2 $ diverges, for
some $ \kk=1,2 $, then the right end of the first line in \eqref{eq:coercive}
will diverge positively, leading to a contradiction with
$ \f_\rho < 0 $.
\end{proof}
\section{Solutions on bounded domains}
\label{sect:three}
Given $ R > 0 $, we denote with $ B_R $ the ball centered at the origin 
with radius $ R $. We define the functional
\begin{equation}
\label{eq:jbdd}
\J_R \colon H_0 \sp 1 (B_R)\oplus H_0 \sp 1 (B_R)\ra\R
\end{equation}
as the restriction of $ \J $. We look at the minimizers of $ J_R $
over the constraint
\[
N_{\rho} (B_R) := \{u\in H_0 \sp 1 (B_R)\oplus H_0 \sp 1 (B_R)\,|\,
\no{u_\kk}_{L\sp 2 (B_R)} \sp 2 = \rho_\kk\}.
\]
The assumptions on $ \F $ are those stated in the previous section (even
if some of them could be relaxed).
\begin{prop}
\label{prop:bounded}
The functional $ \J_R $ attains its infimum on $ N_{\rho} (B_R) $. If 
$ \J_R (u) = \inf_{N_\rho (B_R)} \J_R $, then $ u_\kk > 0 $ or $ u_\kk < 0 $
on $ B_R $. Moreover, a minimum of $ J_R $ can be chosen to be positive,
radially symmetric and decreasing, and of class 
$ C\sp{1,\alpha} (\overline{B}) $.
\end{prop}
\begin{proof}
In order to simplify the notation, we denote $ \J_R $ with $ \J $, $ B_R $ 
with $ B $ and $ N_\rho (B_R) $ with $ N $. Let $ L $ be the Lagrangian
associated to $ \J $
\begin{gather*}
L\colon B\times\R\sp 2\times\R\sp{2\N}\ra\R\\
(x,z,p)\mapsto\frac{1}{2} |p|\sp 2 -\beta|z_1 z_2|\sp\gamma + \G(z)
\end{gather*}
First, we observe that $ \J $ is weakly lower semi-continuous. Let 
$ \lambda\in\R $ and $ 2\sp * > r > 2\gamma $ be such that 
\[
\begin{split}
L_{\lambda} (x,z,p) := L(x,z,p) + \lambda (|z|\sp 2 + |z|\sp r)\geq 0
\end{split}
\]
Such $ \lambda $ and $ r $ exist from the hypothesis on $ \gamma $
in \eqref{eq:A1} and \eqref{eq:A2}. Because $ L $ is convex in $ p $,
also $ L_{\lambda} $ is convex in $ p $. We denote with $ \J_{\lambda} $
the functional associated to $ L_{\lambda} $. Thus, $ \J_{\lambda} $
is weakly lower semi-continuous by \cite[Theorem~1.6,p.\,9]{Str08}.
Given a weakly converging sequence $ u_n \wk u $ in $ \HH $, there exists
a subsequence $ u_{n_k} $ such that
\begin{gather*}
u_{n_k}\ra u\in L\sp 2 (B_R)\cap L\sp r (B_R)\\
\liminf_{n\ra\infty} \J(u_n) = \lim_{k\ra\infty} \J(u_{n_k})
\end{gather*}
from the Rellich-Kondrakov theorem, 
\cite[Th\'eor\`eme~IX.16,p.\,169]{Bre83}. We have
\[
\begin{split}
\liminf_{n\ra\infty}\J(u_n) =& \lim_{k\ra\infty} \J(u_{n_k})\\
=& -\lambda\left(\no{u_{n_k}}\sp 2 + \no{u_{n_k}}_r \sp r\right) + 
\lim_{k\ra\infty} \J_{\lambda} (u_{n_k})\\
= &
-\lambda\left(\no{u}\sp 2 + \no{u}_r \sp r\right) + 
\liminf_{k\ra\infty} \J_{\lambda} (u_{n_k})\\
\geq & -\lambda\left(\no{u}\sp 2 + \no{u}_r \sp r\right) 
+ \J_{\lambda} (u)
= \J(u).
\end{split}
\]
Let $ u_n $ be a minimizing sequence for $ J $ over $ N $. By (ii) of
Lemma~\ref{lem:general} such sequence is bounded. Thus, by the
Rellich-Kondrakov theorem, we can suppose that 
$ u_n \ra u $ in $ L\sp 2 (B) $ and weakly
in $ H\sp 1 _0 (B) $. Then $ u\in N $.
Because $ J $ is weakly lower semi-continuous, we have
\[
J(u)\leq\liminf_{n\ra +\infty} J(u_n) = \inf_N J.
\]
Hence $ u $ is a minimizer of $ J $. We argue by contradiction and
suppose that $ u $ vanishes at some point. Let
\[
v_\kk := |u_\kk|\geq 0.
\]
It is easy to check that $ v\in N $ and $ J(u) = J(v) $. Then $ v $
is a weak solution of the 
\begin{equation}
\label{eq:prop:bounded-1}
-\Delta v_\kk = \lambda_\kk v_\kk
+ \cg\gamma v_\kk\sp{\gamma - 1} v_{\sigma(\kk)}\sp\gamma - \D_\kk\G(v)
\end{equation}
for some $ \lambda_\kk\in\R $ and where 
$ \sigma(1) = 2 $ and $ \sigma(2) = 1 $.
By local regularity theory $ v_\kk\in H\sp 2 _{loc} (B)\cap C(\overline{B}) $
and $ v_\kk $ vanishes at some point. We have
\[
- \Delta v_\kk - \lambda_\kk v_\kk + D_\kk \G (v)\geq 0.
\]
By (\ref{eq:A1},\ref{eq:A2}), we have a
well defined function
\[
A_\kk (x) =  
\begin{cases}
\lambda_\kk - \D_\kk \G (v) v_\kk \sp{-1} & \text{ if } v_\kk (x)\neq 0\\
\lambda_\kk &\text{ otherwise.}
\end{cases}
\]
Therefore
\[
\Delta v_\kk + A_\kk \sp - (x) v_\kk\leq 0.
\]
Because $ v_\kk $ is bounded on $ B $ and $ \D_\kk\G $ is continuous, we have 
$ A_\kk\sp-\in L\sp\infty (B) $. Thus we can apply the
strong maximum principle: if $ v_\kk $ vanishes in the 
interior of $ B $, by \cite[Theorem~3.5]{GT98}, $ v_\kk\equiv 0 $.
Because this is not possible by the constraint condition, we obtain 
$ v_\kk > 0 $, a contradiction.\vskip .2em
\noindent Now, given a positive minimizer $ u $, we can take the decreasing
rearrangement $ u\sp * $. By \cite[Eq.~(4),p.\,81]{LL01}
\[
u\sp *\in N_\rho (B).
\]
Moreover, by \cite[Lemma~7.17,p.\,188]{LL01},
\cite[Lemma~2.1]{Kaw85} and assumption~\eqref{eq:A4}, it follows that 
\[
J(u\sp *)\leq J(u).
\]
In fact, due to the minimization property of $ u $, the inequality is
an equality, hence $ u\sp * $ is a minimizer.
Because of the radial symmetry, we have 
$ u\sp *\in C\sp{1,\alpha} (\overline{B}) $ for some $ \alpha\in (0,1) $.
\end{proof}
\section{The sub-additivity property of $ \f $}
\label{sect:subadd}
The next Lemma is the one-dimensional case of \cite[Proposition~1.4]{Bye00}.
We include the proof because, due to the specifity of the case, we can
state a more precise inequality. We use the notation $ u\sp * $ for
the symmetric decreasing rearrangement.
\begin{lemma}
\label{lem:steiner}
Let $ u,v\in H\sp 1 (\R) $ be two compactly supported, symmetric functions
with respect to the origin such that
\[
\supp(u) = [-c,c],\quad\supp(v) = [-d,d]
\]
and $ \sup(u)\leq\sup(v) $. Moreover, $ u $ and $ v $ are differentiable
except on the boundary and
\begin{equation}
\label{eq:decreasing}
t u'(t),t v'(t) < 0
\end{equation}
on the complementary of a finite subset.
Let $ T $ be such that
\[
\supp(u)\cap\supp(v(\cdot - T)) = \emptyset.
\]
We define $ w(t) := u(t) + v(t - T) $. Then 
\begin{equation}
\label{eq:lem:steiner}
\no{{w\sp *}'}\sp 2 \leq\no{w'}\sp 2 - \frac{3}{4}\no{u'}\sp 2.
\end{equation}
\end{lemma}
\begin{proof}
We set $ a := \sup(u) $ and $ b := \sup(v) $. The functions
\[
u\colon (0,c)\ra (0,a),\quad v\colon (0,d)\ra (0,b)
\]
are invertible, because they are strictly decreasing, by \eqref{eq:decreasing}.
Let $ y_u $ and $ y_v $ be these inverses. Thus,
\begin{equation}
\label{eq:lem:steiner-1}
u(y_u (s)) = s\text{ on } (0,a),\quad v(y_v (s)) = s\text{ on } (0,b)
\end{equation}
Because $ w\sp * $ is symmetric and decreasing, the level set
$ \{w\sp * > s\} $ is an interval. We define its width by $ 2z(s) $.
We have
\begin{equation}
\label{eq:lem:steiner-2}
2z(s) = |\{w\sp * > s\}| = 
\begin{cases}
2y_u (s) + 2y_v(s) &\text{ if } s\in (0,a)\\
2 y_v (s) &\text{ if } s\in (a,b).
\end{cases}
\end{equation}
The second equality follows from the definition of decreasing rearrangement.
Because $ y_u $ and $ y_v $ are strictly decreasing functions and
differentiable everywhere, so is $ z $. Moreover
\begin{equation}
\label{eq:lem:steiner-3}
w\sp* (z(s)) = s\text{ on } (0,b).
\end{equation}
Taking the derivative with respect to $ s $ in \eqref{eq:lem:steiner-3}
and in \eqref{eq:lem:steiner-1}, we have
\begin{equation}
\label{eq:lem:steiner-4}
{w\sp *}' (z(s)) z' (s) = 1,\quad u' (y_u (s)) y_u ' (s) = 1,
\quad v' (y_v (s)) y_v ' (s) = 1.
\end{equation}
Hence
\begin{equation}
\label{eq:lem:steiner-5}
\begin{split}
\int_\R |{w\sp *}'|\sp 2 dt &= 2\int_0 \sp{c + d} |{w\sp *}'|\sp 2 \\
&=-2\int_0 \sp b |{w\sp *}' (z(s))|\sp 2 z'(s) ds
= -2\int_0 \sp b (z'(s))\sp{-1} ds \\
&= - 2 \int_0 \sp a (y' _u (s) + y' _v (s))\sp{-1} 
- 2\int_a \sp b (y' _v (s))\sp{-1} ds.
\end{split}
\end{equation}
The second equality follows from a change of variable and 
\eqref{eq:lem:steiner-4}. The fourth equality follows from 
\eqref{eq:lem:steiner-2}. We use the inequality
\[
2(x + y)\sp{-1}\leq x\sp{-1} + y\sp{-1} - \max\{x\sp{-1},y\sp{-1}\}.
\]
Thus, the last term of \eqref{eq:lem:steiner-5} is bounded from above by
\[
\begin{split}
& -\int_0 \sp a (y' _u (s))\sp{-1} + (y' _v (s))\sp{-1}ds \\
&+ \int_0 \sp a \max\{y' _u (s)\sp{-1},y' _v (s)\sp{-1}\} ds
- 2\int_a \sp b (y' _v (s))\sp{-1} ds
\end{split}
\]
using the estimate $ 2\max\{t,s\}\geq t + s $, the last term is bounded by
\[
\begin{split}
&-\frac{1}{2}\int_0 \sp a (y'_u (s))\sp{-1} ds
-\frac{1}{2}\int_0 \sp a (y' _v (s))\sp{-1} ds
-2\int_a \sp b (y'_v (s))\sp{-1} ds\\
\leq& -\frac{1}{2}\int_0 \sp a (y'_u (s))\sp{-1} ds
-2\int_0 \sp b (y'_v (s))\sp{-1} ds \\
=& \frac{1}{4} \cdot\left( - 2 \int_0 \sp a (y'_u (s))\sp{-1}\right)
+ \left(-2\int_0 \sp b (y'_v (s))\sp{-1}\right) ds.
\end{split}
\]
From a change of variable and \eqref{eq:lem:steiner-4} it follows
\[
\no{u'}\sp 2 = - 2 \int_0 \sp a (y'_u (s))\sp{-1} ds,
\quad
\no{v'}\sp 2 = -2\int_0 \sp b (y'_v (s))\sp{-1} ds.
\]
Thus, from \eqref{eq:lem:steiner-5}, we obtain
\[
\no{{w\sp *}'}\sp 2 \leq \frac{1}{4}\no{u'} \sp 2 + \no{v'}\sp 2 =
\no{w'}\sp 2 - \frac{3}{4}\no{u'}\sp 2.
\]
\end{proof}
\begin{prop}
\label{lem:subadd}
Let $ \rho,\tau $ be such that $ \rho_i\geq\tau_i > 0 $ and $ \tau\neq\rho $.
Then,
\[
\f_\rho < \f_\tau + \f_{\rho - \tau}
\]
that is, $ f $ is sub-additive.
\end{prop}
\begin{proof}
Let us define $ \sigma := \rho - \tau $ and let
\begin{gather}
u_n\in N_\tau,\quad v_n\in N_\sigma
\end{gather}
be minimizing sequences of $ \J $ over $ N_\tau $ and $ N_\sigma $, 
respectively. Because $ \J $ is continuous by (a) of 
Proposition~\ref{prop:smooth}, we can suppose that all the 
functions above have compact support. Thus, there exists a ball 
$ B_n\subset\R\sp\N $ such that
\[
\supp(u_n)\cup\supp(v_n)\subseteq B_n.
\]
By Proposition~\ref{prop:bounded}, we can replace $ u_n $ and
$ v_n $ with the corresponding minimizers of $ \J $ over 
$ N_{\tau} (B_n) $ and $ N_{\sigma} (B_n) $. By
the same proposition, these can be chosen to be positive in the interior of 
$ B_n $, radially symmetric and $ C\sp 1 (\overline{B_n}) $. 
Using the Stone-Weierstrass theorem, we can suppose that
\begin{gather*}
u_n \sp\kk (x) := p_n \sp\kk (|x|)
\end{gather*}
where $ p_n $ is a polynomial. Hence,
\begin{equation}
\label{eq:nice}
u_n \sp\kk (x',\cdot)' (t) = 
\frac{t {p_n \sp\kk}' (|(t,x')|)}{\sqrt{t\sp 2 + |x'|\sp 2}}.
\end{equation}
Thus, $ u_n \sp\kk (x',\cdot)' $ vanishes in a finite number of points,
for every $ x'\in\R\sp{\N - 1} $.
Because $ u_n \sp\kk $ and $ v_n \sp\kk $ have compact support,
there exists a real sequence $ (t_n) $ such that the two functions
\[
u_n \sp\kk,v_n \sp\kk (\cdot + t_n e_\N)
\]
have disjoint support, where $ e_\N = (0,\dots,0,1) $. Then, we
can apply Lemma~\ref{lem:steiner} to 
\begin{gather}
w_n := u_n + v_n (\cdot + t_n e_\N)\in N_\rho\\
\label{eq:split}
J(w_n) = J(u_n) + J(v_n).
\end{gather}
We denote with $ w_n \sp{*e_\N} $ the Steiner symmetrization of $ w_n $ 
with respect to $ e_\N $.  The Steiner symmetrization has the
same properties of the decreasing rearrangement 
we used in Section~\ref{sect:three}. That is,
\[
\|w_n \sp{\kk\,*e_\N}\| = \no{w_n\sp \kk},\quad 
J(w_n \sp{*e_\N})\leq J(w_n).
\]
Given $ x'\in\R\sp{\N - 1} $, the relation between the Steiner and the
decreasing rearrangement gives
\[
\D_\N w_n\sp{\kk\,*e_\N} (x',t) = {w_n \sp{\kk\,*} (x',\cdot)}' (t)
\]
Then, we can write
\begin{equation}
\label{eq:lem:subadd-1}
\begin{split}
&\int_{\R\sp\N} |\D_\N w_n \sp{\kk\,*e_\N}|\sp 2 =
\int_{\R\sp{\N - 1}} \int_\R |{{w_n \sp{\kk\,*} (x',\cdot)}' (t)}|\sp 2 
dt\,dx' \\
=&
\int_{U_n \sp\kk} \int_\R |{{w_n \sp{\kk\,*} (x',\cdot)}' (t)}|\sp 2 
dt\,dx'
\int_{V_n \sp\kk} \int_\R |{{w_n \sp{\kk\,*} (x',\cdot)}' (t)}|\sp 2 
dt\,dx'
=: A_1 \sp\kk + A_2 \sp\kk
\end{split}
\end{equation}
where
\begin{align*}
U_n \sp\kk &= \{x'\in\R\sp{\N - 1}\,|\,\sup_{\R} u_n \sp\kk (x',\cdot) \leq
\sup_{\R} v_n \sp\kk (x',\cdot)\}\\
V_n \sp\kk &= \{x'\in\R\sp{\N - 1}\,|\,\sup_{\R} v_n \sp\kk (x',\cdot) <
\sup_{\R} u_n \sp\kk (x',\cdot)\}.
\end{align*}
For every $ x'\in\R\sp{\N - 1} $, $ u_n \sp\kk (x',\cdot) $ and 
$ v_n \sp\kk (x',\cdot + te_\N) $
satisfy the hypothesis of Lemma~\ref{lem:steiner}, by 
Proposition~\ref{prop:bounded} and \eqref{eq:nice}. Thus,
\begin{align*}
A_1 \sp \kk &\leq\int_{U_n \sp \kk } 
\Big(\no{{w_n (x',\cdot)}'}_{L\sp 2 (\R)} \sp 2 -
\frac{3}{4}\no{{u_n \sp \kk (x',\cdot)}'}_{L\sp 2 (\R)}\sp 2\Big)
dx'\\
A_2 \sp \kk &\leq\int_{V_n \sp \kk} 
\Big(\no{{w_n \sp \kk (x',\cdot)}'}_{L\sp 2 (\R)} \sp 2 -
\frac{3}{4}\no{{v_n \sp \kk (x',\cdot)}'}_{L\sp 2 (\R)}\sp 2\Big) dx'.
\end{align*}
In the last equalities we used
\[
\no{v_n \sp\kk (x',\cdot)}_{L\sp 2 (\R)} = 
\no{v_n \sp\kk (x',\cdot + t_n e_\N)}_{L\sp 2 (\R)}
\]
and the fact that the definition of $ U_n \sp\kk $ and $ V_n \sp\kk $
do not change if we shift the last coordinate.
Thus,
\begin{equation}
\label{eq:lem:subadd-3}
\begin{split}
A_1 \sp \kk + A_2\sp\kk &\leq\no{\D_\N w_n \sp\kk}\sp 2 \\
&-
\frac{3}{4}\left(
\no{\D_\N u_n \sp\kk}_{L\sp 2 (U_n \sp\kk\times\R)}\sp 2
+ \no{\D_\N v_n\sp\kk}_{L\sp 2 (V_n \sp\kk\times\R)}\sp 2\right).
\end{split}
\end{equation}
Because $ u_n $ is radially symmetric, then
\begin{gather*}
\D_\ii u_n \sp\kk (x_1,\dots,x_\ii,\dots,x_\N) = 
\D_\N u_n \sp\kk (x_1,\dots,x_\N,\dots,x_\ii)
\end{gather*}
for every $ 1\leq\ii\leq\N $. Then, for every measurable subset 
$ A\subset\RN $, by the equality above
\[
\no{\D_\N u_n\sp\kk}_{L\sp 2 (A)} = \no{\D_\ii u_n\sp\kk}_{L\sp 2 (A_i)}
\]
where $ A_i $ is the set obtained by $ A $ throught the permutation
of the coordinates $ \ii $ and $ \N $. Moreover,
\[
\no{\D_\ii u_n\sp\kk}_{L\sp 2 (A_i)} \sp 2 = 
\no{\D_\ii u_n\sp\kk}_{L\sp 2 (A)} \sp 2.
\]
Because $ \D_\ii u_n\sp\kk $ is radially symmetric.
In conclusion 
\eqref{eq:lem:subadd-3} can be written as
\begin{equation}
\label{eq:lem:subadd-20}
\begin{split}
\N(A_1 \sp \kk + A_2\sp\kk) &\leq\no{\D w_n \sp\kk}\sp 2 \\ 
&- 
\frac{3}{4}\left(%
\no{\D u_n \sp\kk}_{L\sp 2 (U_n \sp\kk\times\R)}\sp 2 %
+ \no{\D v_n\sp\kk}_{L\sp 2 (V_n \sp\kk\times\R)}\sp 2\right)
\end{split}
\end{equation}
We define
\begin{equation}
\label{eq:lem:subadd-9}
d_n \sp \kk = \no{\D u_n \sp\kk}_{L\sp 2 (U_n \sp \kk\times\R)}\sp 2
+ \no{\D v_n\sp\kk}_{L\sp 2 (V_n \sp \kk\times\R)}\sp 2.
\end{equation}
We will prove that a subsequence of $ (d_n \sp\kk) $ is bounded
from below. Because $ u_n $ and $ v_n $ are minimizing sequences,
from \eqref{eq:young} and (i) of Lemma~\ref{lem:general} 
\begin{equation}
\label{eq:lem:subadd-4}
\no{u_n \sp \kk}_{2\gamma},\no{v_n \sp \kk}_{2\gamma}\geq c = 
c(\rho,\tau) > 0.
\end{equation}
Because they are radially symmetric, by \cite[Theorem~A.I']{BL83},
up to extract a subsequence we can suppose that
\[
u_n \sp\kk\ra u_\kk,v_n \sp\kk\ra v_\kk\text{ in } L\sp{2\gamma} (\RN)
\]
and $ u_\kk,v_\kk\not\equiv 0 $ by \eqref{eq:lem:subadd-4}. 
Up to extract a subsequence, we can suppose that the convergence is
pointwise a.e. Thus, there are points $ x_\kk,y_\kk $ other than $ 0 $
such that
\begin{equation}
\label{eq:lem:subadd-8}
u_\kk (x_\kk),v_\kk (y_\kk)\geq h > 0.
\end{equation}
Then for $ n\geq n_1 $, 
\begin{equation}
\label{eq:lem:subadd-7}
u_n \sp\kk (x_\kk),v_n \sp\kk (y_\kk)\geq h/2.
\end{equation}
We define $ R = \min\left\{|x_\kk|,|y_\kk|\,|\,\kk=1,2\right\} $. Because
$ u_n \sp\kk $ and $ v_n \sp\kk $ are radially decreasing, for
every $ n\geq n_1 $ and $ x\in B(0,r) $, we have
\begin{equation}
\label{eq:lem:subadd-5}
u_n \sp\kk (x)\geq h/2,\quad
v_n \sp\kk (x)\geq h/2.
\end{equation}
By applying the second inequality of \eqref{eq:sob-3} to the domains
\[
(U_n \sp i \times\R)\cap B_r,\quad 
(V_n \sp i \times\R)\cap B_r,
\]
because $ u_n\in N_\tau $ and $ v_n\in N_\sigma $, and 
$ \rho_\kk\geq\sigma_\kk,\tau_\kk $, from \eqref{eq:lem:subadd-5}, 
there exists $ c_\kk = c_\kk (\rho_\kk,S) > 0 $ such that
\[
\begin{split}
\no{\D u_n \sp\kk}_{L\sp 2 (U_n \sp\kk \times\R)}
&\geq\no{\D u_n \sp\kk}_{L\sp 2 ((U_n \sp\kk \times\R)\cap B_r)}\\
&\geq c\left(%
\frac{h|(U_n \sp\kk \times\R)\cap B_R|\sp{\frac{1}{2\gamma}}}%
{2}\right)\sp{\frac{2\gamma}{n(\gamma - 1)}}
\end{split}
\]
and
\[
\begin{split}
\no{\D v_n \sp\kk}_{L\sp 2 (V_n \sp\kk \times\R)}\geq&
\no{\D v_n \sp\kk}_{L\sp 2 ((V_n \sp\kk \times\R)\cap B_R)}\\
\geq& c \left(%
\frac{h|(V_n \sp\kk \times\R)\cap B_R|\sp{\frac{1}{2\gamma}}}%
{2}\right)\sp{\frac{2\gamma}{n(\gamma - 1)}}
\end{split}
\]
Because $ U_n \sp\kk\cap B_R $ is the complementary of 
$ V_n \sp\kk\cap B_R $ in $ B_R $, for every $ \kk=1,2 $, at least one of 
the two quantities
\[
\delta_\kk := 
\liminf_{n\ra\infty} \frac{|(U_n \sp\kk\times\R)\cap B_R|}{|B_R|},\
1 -\delta_\kk
\]
is not smaller that $ 1/2 $. Then,
\[
d_n \sp\kk\geq
2\sp{-\frac{2(2\gamma + 1)}{n(\gamma - 1)}}
h\sp{\frac{4\gamma}{n(\gamma - 1)}}
|B_R|\sp{\frac{1}{n(\gamma - 1)}} c_\kk
=: d_\kk
\]
for $ \kk=1,2 $. Then, from (\ref{eq:lem:subadd-1},\ref{eq:lem:subadd-20})
and the inequality above we obtain
\begin{equation}
\label{eq:lem:subadd-10}
\N \int_{\R\sp\N} |\D_\N w_n \sp{\kk\,*e_\N}|\sp 2\leq
\no{\D w_n \sp\kk}\sp 2 - \frac{3d_n \sp\kk}{4}\leq 
\no{\D w_n \sp\kk}\sp 2 - \frac{3d_\kk}{4}.
\end{equation}
Finally, we consider the decreasing rearrangement of $ w_n \sp{*e_\N} $ 
(which may not be radially symmetric). By \cite[Lemma~11]{Kaw85}, when 
$ \N = 1 $, we have
\[
\begin{split}
\no{\D_n w_n \sp{\kk\,*e_\N*}}\sp 2 &= \int_{\R\sp{\N - 1}}
\no{w_n \sp{\kk\,*e_\N*} (x',\cdot)'}_{L\sp 2 (\R)}\sp 2 dx'\\
&\leq
\int_{\R\sp{\N - 1}}
\no{w_n \sp{\kk\,*e_\N} (x',\cdot)'}_{L\sp 2 (\R)}\sp 2 dx' =
\no{\D_n w_n \sp{\kk\,*e_\N}}\sp 2.
\end{split}
\]
From \eqref{eq:lem:subadd-10}, we can write
\[
\N\int_{\R\sp\N} |\D_\N w_n \sp{\kk\,*e_\N*}|\sp 2\leq 
\no{\D u_n\sp\kk}\sp 2 + \no{\D v_n\sp\kk}\sp 2 - 
\frac{3d_\kk}{4}.
\]
Because $ w_n \sp{\kk\,*e_\N*} $ is radially symmetric
\[
\int_{\R\sp\N} |\D w_n \sp{\kk\,*e_\N*}|\sp 2\leq 
\no{\D u_n\sp\kk}\sp 2 + \no{\D v_n\sp\kk}\sp 2 - \frac{3d_\kk}{4}.
\]
Hence,
\[
\f_\rho\leq J(w_n \sp{*e_\N*})\leq J(u_n) + J(v_n) - \frac{3d}{8}
\]
where $ d := d_1 + d_2 $. Taking the limit as $ n\ra\infty $, we obtain
\[
\f_\rho\leq \f_\tau + \f_\sigma - D
\]
where $ D := 3d/8 > 0 $.
\end{proof}
\section{Concentration of minimizing sequences in $ \R\sp\N $}
In this section we establish the existence of a minimizer for $ J $
on $ \R\sp\N $ and a concentration property of minimizing sequences
on the constraint $ N_\rho $. 
We start with the following
\begin{lemma}
\label{lem:tied-together}
Let $ (u_n)_{n\geq 1} $ be a bounded sequence in $ \HH $ such that
\[
\liminf_{n\ra\infty}\nint |u_n \sp 1 u_n \sp 2|\sp\gamma > 0
\]
where $ 1 <  \gamma < 2\sp */2 $. Then, there exists a sequence 
$ (y_n)\subset\RN $ such that
\[
u_n \sp\kk (\cdot - y_n)\wk u_\kk,\quad u_1 u_2\not\equiv 0.
\]
\end{lemma}
\begin{proof}
Let $ w_n = u_n \sp 1 u_n \sp 2 $. From the H\"older inequality and
\eqref{eq:sob-2}, it follows that
\[
Dw_n\in L\sp{\N/\N - 1}.
\]
We apply \cite[Lemma~I.1,p.\,231]{Lio84-II} with $ q = 1 $ and 
$ p = \N/\N - 1 $. Hence, either there exists $ R > 0 $ and a sequence
$ (y_n) $ such that
\begin{equation}
\label{eq:not-vanish}
\liminf_{n\ra\infty}\int_{B(-y_n,R)} |w_n| > 0
\end{equation} 
or 
\begin{equation}
\label{eq:contr}
w_n\ra 0\text{ in } L\sp\alpha,\text{ for every }
\alpha\in (1,\N/\N - 2).
\end{equation}
The latter cannot happen because, by \eqref{eq:A1}
\[
\gamma < \frac{2\sp *}{2} = \frac{\N}{\N - 2}.
\]
So, \eqref{eq:contr} would contradict the hypothesis of the Lemma. Hence
\eqref{eq:not-vanish} holds.
By changing the variable of integration in \eqref{eq:not-vanish}
and letting 
\[
v_n \sp \kk = u_n \sp \kk (\cdot - y_n)
\]
we obtain 
\begin{equation}
\label{eq:bdd}
\liminf_{n\ra\infty}\int_{B_R} v_n \sp 1 v_n \sp 2 > 0.
\end{equation}
Since $ v_n \sp\kk $ are bounded in $ H\sp 1 $, we can suppose that they
converge weakly to some limits $ u_1 $ and $ u_2 $, respectively. By the
Rellich-Kondrakhov theorem, we can suppose that such convergence is
strong in $ L\sp 2 (B_R) $. Thus, in \eqref{eq:bdd} we can take the limit
in the integrand and obtain
\[
\int_{B_R} u_1 u_2 > 0
\]
which implies $ u_1 u_2\not\equiv 0 $. And
$ u_n \sp\kk (\cdot - y_n)\rightharpoonup u_\kk $.
\end{proof}
\begin{theorem}
\label{thm:concentration}
Let $ (u_n)_{n\geq 1} $ be a minimizing sequence
for $ J $ over $ N_\rho $. Then, there
exists $ u\in N_\rho $ and a sequence 
$ (y_n)_{n\geq 1} $ such that
\begin{gather*}
u_n = u(\cdot + y_n) + o(1)\text{ in } \HH\\
J(u) = \inf_{N_\rho} J.
\end{gather*}
\end{theorem}
\begin{proof}
By (i) and (ii) of Lemma~\ref{lem:general} $ \f_\rho < 0 $ and the
sequence $ (u_n)_{n\geq 1} $ is bounded. 
Because $ G\geq 0 $ the sequence $ (u_n) $ fulfils the hypothesis of 
Lemma~\ref{lem:tied-together}, once $ \gamma < \N/(\N - 2) $ holds. 
This, in turn, follows from \eqref{eq:A1} and 
\[
1 + \frac{2}{\N} < \frac{\N}{\N - 2}.
\]
Then, we consider the sequence $ (y_n)_{n\geq 1} $ and $ u\in\HH $ 
given by the Lemma~\ref{lem:tied-together}. We define
\begin{gather*}
v_n := u_n (\cdot - y_n) - u,\ \ 
\tau := \big(\no{u_1}_{L\sp 2} \sp 2,\no{u_2}_{L\sp 2} \sp 2\big).
\end{gather*}
We have $ \tau_i\leq\rho_i $, by the weak lower semi-continuity property of
the $ L\sp 2 $-norm. Suppose that $ \tau\neq\rho $. By (b) of 
Proposition~\ref{prop:smooth}, up to extract a subsequence, 
we have
\begin{equation}
\label{eq:cont}
J(v_n) = J(u_n(\cdot +y_n)) - J(u) + o(1).
\end{equation}
By a change of variable, the first term of the right member equals
$ J(u_n) $ which converges to $ \f_\rho $. Hence, by Lemma~\ref{lem:subadd}
\[
\f_{\rho - \tau}\leq \f_\rho - \f_\tau < \f_{\rho - \tau} - D
\]
for some constant $ D > 0 $. Thus $ \tau = \rho $ and $ u\in N_\rho $.
Moreover,
\[
u_n (\cdot - y_n) - u \ra 0\text{ in } L\sp 2 (\RN,\R\sp 2).
\]
By the weak lower semi-continuity property of $ J $,
$ J(u_n)\ra J(u) $. This completes the first and the third statement 
of the Theorem. In order to complete the proof, we only need to show that
the convergence above holds in $ \HH $ as well. Because
\begin{equation}
\label{eq:thm:concentration+1}
(u_n (\cdot + y_n))_{n\geq 1}
\end{equation}
is a minimizing sequence, by Ekeland's Theorem 
\cite[Theorem~5.1,p.\,51]{Str08} there exists a sequence $ w_n $ 
such that 
\[
\no{w_n - u_n (\cdot + y_n)}_{H\sp 1}\ra 0.
\]
and $ (w_n)_{n\geq 1} $ is Palais-Smale. That is, there are 
$ \lambda_1,\lambda_2\in\R $ such that, 
\[
X_n := \D\J(w_n) - (\lambda_1 w_n \sp 1,\lambda_2 w_n \sp 2)\ra 0.
\]
Then
\begin{equation}
\label{eq:thm:concentration}
\begin{split}
\sum_{\kk = 1} \sp 2 (X_n \sp\kk,w_n \sp\kk - w_m \sp\kk) &= 
\sum_{\kk = 1} \sp 2 \no{\D w_n \sp\kk - \D w_m \sp\kk} \sp 2 - 
\lambda_\kk \no{w_n \sp\kk - w_m \sp\kk}\sp 2 \\
&+\nint (\D_\kk \F (w_n) - \D_\kk \F (w_m))(w_n \sp\kk - w_m \sp\kk).
\end{split}
\end{equation}
We look at the third summand of the right-end of the equality above.
We have
\begin{align*}
\D_1 \F (w_n) &= -\gamma\cg w_n \sp{1,\gamma - 1} w_n \sp{2,\gamma} +
\D_1 \G (w_n)\\
\D_2 \F (w_n) &= -\gamma\cg w_n \sp{2,\gamma - 1} w_n \sp{1,\gamma} +
\D_2 \G (w_n).
\end{align*}
By \eqref{eq:A2}, we have
\begin{equation}
\label{eq:thm:concentration-2}
\begin{split}
\left|\nint D_\kk \G (w_n)(w_n \sp\kk - w_m \sp\kk)\right|&\leq
\co\nint |w_n| \sp{p - 1} |w_n \sp\kk - w_m \sp\kk| \\
&+ \co\nint |w_n| \sp{q - 1} |w_n \sp\kk - w_m \sp\kk|.
\end{split}
\end{equation}
We apply the H\"older inequality to each of the integrands with pairs 
$ (p',p) $ and $ (q',q) $ respectively. From inequality \eqref{eq:sob-3}
where $ 2\gamma $ is replaced by $ p $ and $ q $, it follows
\[
\begin{split}
\nint |w_n| \sp{p - 1} |w_n \sp\kk - w_m \sp\kk| &\leq
\no{w_n}_p \sp{p - 1} \no{w_n \sp\kk - w_m \sp\kk}_p\\
&\leq c \no{w_n \sp\kk - w_m \sp\kk}\sp\theta
\end{split}
\]
for some $ c > 0 $ and $ \theta\in (0,1) $. 
A similar inequality holds for $ q $. Then the left member in 
\eqref{eq:thm:concentration-2} converges to zero for $ i=1,2 $.
We use the H\"older inequality to estimate the term
\begin{equation}
\label{eq:thm:concentration-3}
\gamma \cg\nint w_n \sp{1,\gamma - 1} w_n \sp{2,\gamma} 
|w_n \sp\kk - w_m \sp\kk|
\end{equation}
as well. In order to do so, we need to find a triple of real numbers
$ (r_1,r_2,r_3) $ such that
\begin{gather*}
r_1 (\gamma - 1),r_2\gamma\in [2,2\sp *],r_3\in [2,2\sp *)\\
\sum_{i = 1} \sp 3 \frac{1}{r_i} = 1,\quad r_i\geq 1.
\end{gather*}
We briefly check that we can achieve such a triple. The first line can
be written as
\begin{gather*}
r_1\sp{-1} \in [(\gamma - 1)/2\sp *,
(\gamma - 1)/2],
\quad
r_2\sp{-1}\in [\gamma/2\sp*,\gamma/2]\\
1 - r_1 \sp{-1} - r_2 \sp{-1}\in (1/2\sp*,1/2].
\end{gather*}
From \eqref{eq:A1}, $ 2/\gamma - 1 > 2/\gamma > 1 $ for $ \N\geq 3 $. Thus,
the requirement $ r_i\geq 1 $ is included in the first of the two lines
above. Then, it is enough to check that
\[
[1 - (2\gamma - 1)/2, 1 - (2\gamma - 1)/2\sp *]\cap (1/2\sp *,1/2]
\neq\emptyset.
\]
In fact, the intersection above is empty if and only if either
\[
1 - \frac{2\gamma - 1}{2\sp *} \leq\frac{1}{2\sp*}\Rightarrow
\gamma\geq\frac{\N}{\N - 2} > 1 + \frac{2}{\N}
\]
or
\[
\frac{1}{2} < 1 - \frac{2\gamma - 1}{2}\Rightarrow \gamma < 1.
\]
Both of them contradict \eqref{eq:A1}. Hence \eqref{eq:thm:concentration-3}
can be estimated from above by
\[
c\no{w_n \sp\kk - w_m \sp\kk}\sp{\theta}
\]
for some $ c > 0 $ and $ \theta\in (0,1) $. 
Thus, from \eqref{eq:thm:concentration}$ 
(\D w_n \sp\kk)_{n\geq 1} $ is a Cauchy sequence and
\[
w_n \sp\kk \ra u_\kk,\ \ \D w_n \sp\kk\ra f_\kk
\]
for some $ f_\kk\in L\sp 2 (\RN) $. Hence $ f_\kk = \D u_\kk $ and
$ w_n \ra u $ in $ \HH $. By \eqref{eq:thm:concentration+1}
\[
u_n(\cdot - y_n) - u\ra 0 \text{ in } \HH
\]
By a change of variable, the second statement of the Theorem also follows.
\end{proof}
Let $ m_1,m_2 $ such that
\[
\label{eq:A5}
\tag{A5}
F(u) + \frac{1}{2}\big(m_1 \sp 2 u_1 \sp 2 + m_2 \sp 2 u_2 \sp 2)\geq 0.
\]
Under this assumption we prove Theorem~\hyperlink{thm:A}{A}.
\def\proofname{Proof of Theorem~\hyperlink{thm:A}{A}}
\begin{proof}
From \cite[Lemma~1]{Gar10} minimizing sequences of $ E $ over $ \M $
are bounded in $ \HH\oplus\R\sp 2 $. Thus, up to extract a subsequence
\[
\no{u_n \sp\kk}\sp 2 \ra\rho_\kk,\quad\om_n\ra\om.
\]
As in \textsl{Step~II} of the proof \cite[Lemma~2.7]{BBBM08} it can be
shown that 
\[
v_n \sp i = \frac{u_n \sp i \sqrt{\rho_i}}{\no{u_n \sp i}}
\]
is a minimizing sequence for $ \J $ over $ N_\rho $. Then, by 
Theorem~\ref{thm:concentration}, there exists a sequence 
$ (y_n)_{n\geq 1}\subset\RN $ such that
\[
v_n (\cdot + y_n)\ra u\text{ in } \HH
\]
for some $ u\in\HH $. Then, $ (u,\om)\in\M $ is a minimum of $ E $ over
$ \M $.
\end{proof}
\def\proofname{Proof}
\section{Lyapunov functions}
In the space $ \Y := H\sp 1 (\RN,\C\sp 2)\oplus L\sp 2 (\RN,\C\sp 2) $
we denote with $ \dist $ the metric induced by the scalar product
\[
\langle \Phi,\Psi\rangle_{\R} 
:= \re\langle\Phi,\Psi\rangle_{\C}
= \re\sum_{\kk = 1} \sp 2 \nint \phi_\kk \overline{\psi}_\kk 
+ \phi_t \sp\kk \overline{\psi}_t \sp\kk.
\]
\begin{lemma}
\label{lem:inequalities}
Let $ \phi\in H\sp 1 (\RN,\R\sp m) $. Then, $ |\phi|\in H\sp 1 (\RN) $ and the 
inequality
\[
\no{\D\phi}\geq\no{\D|\phi|}.
\] 
holds. If the equality holds and $ |\phi| > 0 $ everywhere, then there exists 
$ \lambda\in\R\sp m $ such that $ |\lambda| = 1 $ and 
\[
\phi(x) = \lambda|\phi(x)|.
\]
\end{lemma}
\begin{proof}
Clearly, $ |\phi|\in L\sp 2 (\RN,\R) $. 
By applying \cite[Theorem~7.4,p.\,150]{GT98} and 
\cite[Theorem~7.8,p.\,153]{GT98} with $ f(\vfi) = |\vfi| $, we obtain
\[
\pt_\ii |\phi| = 
\begin{cases}
\displaystyle\frac{\langle\phi,\pt_\ii \phi\rangle}%
{|\phi|} & \text{ if } \phi\neq 0\\
0 & \text{ if } \phi = 0,
\end{cases}
\]
By the Schwarz inequality
\begin{equation}
\label{eq:10}
\begin{split}
|\D |\phi||\sp 2 &= \sum_{\ii = 1} \sp\N |\pt_\ii |\phi||\sp 2 =
\frac{1}{|\phi|\sp 2}
\sum_{\ii = 1} \sp\N |\langle\phi,\pt_\ii \phi\rangle|\sp 2\\
&\leq\sum_{\ii = 1}\sp n |\pt_\ii \phi|\sp 2 = |\D\phi|\sp 2
\end{split}
\end{equation}
if $ \phi\neq 0 $. When $ \phi = 0 $, the same inequality follows easily.
Therefore, $ \D |\phi|\in L\sp 2 (\RN,\R\sp m) $ and, by integration,
the first part of our statement is proved. Now, we suppose that
the
\[
\no{\D |\phi|}\sp 2 = \no{\D\phi}\sp 2.
\]
Because $ \phi\neq 0 $, by \eqref{eq:10} we obtain
\[
|\phi| |\pt_\ii |\phi|| = |\langle\phi,\pt_\ii \phi\rangle|.
\]
Because $ \phi\neq 0 $, there exists $ \mu_i\colon\RN\ra\R $, such that
\[
\pt_\ii\phi = \mu_\ii \phi.
\]
Thus, for every $ 1\leq\kk\leq m $, we have
\[
\pt_\ii \phi_\kk = \mu_\ii \phi_\kk
\]
We claim that each of the functions 
\[
\lambda_\kk\colon\RN\ra\R,\quad
x\mapsto\frac{\phi_\kk (x)}{|\phi(x)|}
\]
are constants. First, we notice that $ \lambda_\kk\in H\sp 1 _{loc} (\RN) $.
In fact $ \lambda_\kk\in L\sp\infty (\RN) $ and
\[
\pt_i \lambda_\kk = \frac{\pt_\ii \phi_\kk |\phi|\sp 2 - 
\phi_\kk \langle \phi,\pt_\ii\phi\rangle}{|\phi|\sp 2}
\]
hence,
$ |\pt_i\lambda_\kk|\leq 2 |\D\phi|\in L\sp 2 (\RN) $.
Moreover, 
\[
\sum_{h = 1}\sp m \pt_\ii\phi_\kk \phi_h \sp 2 - 
\phi_\kk \phi_h \pt_\ii \phi_h 
= \mu_\ii\sum_h \phi_\kk \phi_h \sp 2 - \phi_\kk \phi_h \sp 2 = 0.
\]
Then $ \lambda_\kk $ is constant. Thus,
\[
\phi_\kk = \lambda_\kk |\phi|\text{ on } \RN.
\] 
We conclude the proof by choosing $ \lambda = (\lambda_1,\dots,\lambda_m) $.
\end{proof}
A similar result of the Lemma above is known in \cite[Theorem~7.8]{LL01}
under the more restrictive hypothesis that $ |\phi_k| > 0 $ on $ \RN $
for some $ 1\leq k\leq m $.
\def\proofname{Proof of Theorem~\hyperlink{thm:B}{B}}
\begin{proof}
Given $ \Phi\in\Gamma_{\Ce} $, there are $ \lambda_1,\lambda_2\in\C $ and 
$ y\in\RN $ such that $ |\lambda_\kk| = 1 $ and
\[
\Phi = (\lambda u(\cdot + y),-i\om\lambda u(\cdot + y)).
\]
Then
\begin{gather*}
\EH(\Phi) = E(u,\om) = m_{\Ce},\quad
\CH_\kk (\Phi) = \om_\kk \no{u_\kk}\sp 2 = \Ce_\kk.
\end{gather*}
Because $ \EH $ and $ \CH_\kk $ are continuous, if 
$ \dist(\Phi_n,\Gamma_\Ce)\ra 0 $, then
\begin{equation}
\label{eq:thm:lyapunov-6}
\EH(\Phi_n)\ra m_{\Ce},\quad\CH_\kk (\Phi_n)\ra \Ce_\kk.
\end{equation}
We prove the converse and suppose that \eqref{eq:thm:lyapunov-6} holds.
Because $ m_{\Ce} > 0 $, we can suppose that $ \phi_n \sp\kk\not\equiv 0 $
for every $ n $.
From the Schwarz inequality, we obtain
\begin{equation}
\label{eq:thm:lyapunov-4}
\frac{\CH_\kk (\phi_n,\phi_n \sp t)}{\no{\phi_n \sp \kk}}
\leq\no{\phi_n \sp{t,\kk}}
\end{equation}
By Lemma~\ref{lem:inequalities} and \eqref{eq:thm:lyapunov-4}
\begin{equation}
\label{eq:thm:lyapunov-5}  
\begin{split}
\EH(\phi_n,\phi_n \sp t) &= \frac{1}{2}\nint |\D\phi_n|\sp 2 + 
|\phi_n \sp t|\sp 2 
+ 2\V(\phi_n)\\
&\geq\frac{1}{2} \nint |\D|\phi_n||\sp 2 + 2\V (|\phi_n \sp 1|,|\phi_n \sp 2|)
+ \frac{1}{2}\sum_{i = 1}\sp 2 \frac{\CH_\kk (\phi_n,\phi_n \sp t)\sp 2}%
{\no{\phi_n \sp\kk}\sp 2}.
\end{split}
\end{equation}
We define
\begin{equation}
\label{eq:thm:lyapunov-3}  
\om_n \sp\kk = \frac{\Ce_n \sp\kk}%
{\no{\phi_n \sp\kk} \sp 2},\quad u_n \sp\kk = |\phi_n \sp \kk|.
\end{equation}
because of $ \CH_\kk (\Phi_n)\ra\Ce_\kk $, the sequence $ \om_n \sp\kk $ 
will eventually become positive for $ n $ large enough. Then
\eqref{eq:thm:lyapunov-5} implies
\begin{equation}
\label{eq:thm:lyapunov-8}
\EH(\phi_n,\phi_n \sp t) 
\geq E(u_n,\om_n)\geq m_\Ce.
\end{equation}
Taking the limit as $ n\ra\infty $, the first of \eqref{eq:thm:lyapunov-6}
implies that $ (u_n,\om_n) $ is a minimizing sequence for 
$ E $ over $ \M $. 
By Theorem~\hyperlink{thm:A}{A}, there exists $ (u,\om)\in\M $
and a subsequence $ (y_n)\subset\RN $ such that
\begin{equation}
\label{eq:thm:lyapunov-14}
|\phi_n| = u(\cdot + y_n) + o(1),\quad\om_n = \om + o(1).
\end{equation}
We set 
\[
\psi_n := \phi_n (\cdot - y_n),\quad\psi_n \sp t := \phi_n \sp t (\cdot - y_n).
\] 
By a change of variable, we have
\begin{equation}
\label{eq:thm:lyapunov-9}
\EH(\psi_n,\psi_n \sp t) = \EH(\phi_n,\phi_n \sp t),\quad
\CH_\kk (\psi_n,\psi_n \sp t) = \CH_\kk (\phi_n,\phi_n \sp t).
\end{equation}
Up to extract a subsequence, we can suppose that there exists
$ (\psi,\psi_t)\in\Y $ such that
\begin{equation}
\label{eq:thm:lyapunov-13}
\psi_n\wk\psi\text{ in } H\sp 1 (\RN,\C\sp 2),\quad
\psi_n \sp t\wk \psi_t\text{ in } L\sp 2 (\RN,\C\sp 2).
\end{equation}
By the weak lower semi-continuity of the norm, the strong convergence
of $ |\psi_n| $ and Lemma~\ref{lem:inequalities}, we have
\[
\begin{split}
\EH(\psi_n,\psi_n \sp t)&=\frac{1}{2}\nint |\D\psi_n|\sp 2 + 
|\psi_n \sp t|\sp 2 + 2 \V(\psi_n)\\
&\geq \frac{1}{2}\nint |\D\psi|\sp 2 + 
|\psi_t|\sp 2 + 2 \V(\psi)\\
&\geq \frac{1}{2}\nint |\D|\psi||\sp 2 + 
2 \V(\psi) + 
\frac{1}{2}\sum_{i = 1}\sp 2 \frac{\CH_\kk (\psi,\psi_t)\sp 2}%
{\no{\psi_\kk}\sp 2} \geq m_{\Ce}.
\end{split}
\]
Taking the limit as $ n\ra\infty $, by 
\eqref{eq:thm:lyapunov-9} and the first of \eqref{eq:thm:lyapunov-6}, 
from the first inequality above, we obtain
\begin{equation}
\label{eq:thm:lyapunov-12}
\lim_{n\ra\infty} \no{\psi_n \sp t} = \no{\psi_t},\quad
\lim_{n\ra\infty} \no{\D\psi_n} = \no{\D\psi}.
\end{equation}
From the second inequality we obtain
\begin{equation}
\label{eq:thm:lyapunov-11}
\nint |\D\psi_\kk|\sp 2 = \nint |\D|\psi_\kk||\sp 2,\quad
\frac{\CH_\kk (\psi,\psi_t)}%
{\no{\psi_\kk}} = \no{\psi_t \sp\kk}.
\end{equation}
We show that $ \psi\equiv u $ almost everywhere. In fact, for every
$ R > 0 $, from the first of \eqref{eq:thm:lyapunov-13}, we have
\[
\psi_n \sp\kk\wk \psi_\kk\text{ in } L\sp 2 (B_R,\C).
\]
Because $ \psi_n \sp\kk $ is bounded in $ H\sp 1 (B_R,\C) $,
up to extract a subsequence, we can suppose that the two
limits
\[
\psi_n\sp\kk\ra g_\kk \text{ in } L\sp 2 (B_R,\C),\quad
|\psi_n\sp\kk|\ra |g_\kk| \text{ in } L\sp 2 (B_R).
\]
hold. By the uniqueness of the weak limit $ g_\kk = \psi_\kk $.
From \eqref{eq:thm:lyapunov-14}, 
\[
|g_\kk| = u_\kk.
\]
Then $ u_\kk\equiv |\psi_\kk| $ on $ B_R $. Taking the limit as $ R\ra\infty $,
we obtain
\begin{equation}
\label{eq:thm:lyapunov-15}
u = |\psi|.
\end{equation}
Thus, $ \no{\psi_n \sp\kk}\ra\no{\psi_\kk} $. From the first of 
\eqref{eq:thm:lyapunov-13}, we obtain
\[
\psi_n\ra\psi\text{ in } L\sp 2 (\RN,\C).
\]
Because $ (u,\om) $ is a minimizer of $ E $ over $ \M $, then $ u_\kk > 0 $. 
This can be achieved with the argument used in Propositon~\ref{prop:bounded}.
Then, $ |\psi_\kk| > 0 $. By 
\eqref{eq:thm:lyapunov-11} and Lemma~\ref{lem:inequalities}, there are
$ \lambda_\kk\in\C $ such that $ |\lambda_\kk| = 1 $ and
\[
\psi_\kk = \lambda_\kk |\psi_\kk| = u_\kk.
\]
The second limit in \eqref{eq:thm:lyapunov-12} and the first in 
\eqref{eq:thm:lyapunov-13} give
\[
\D\psi_n \sp\kk\ra \D\psi_\kk.
\]
Therefore
\begin{equation}
\label{eq:thm:lyapunov-16}
\psi_n \sp\kk\ra\lambda_\kk u_\kk\text{ in } H\sp 1 (\RN,\C).
\end{equation}
The second equality in \eqref{eq:thm:lyapunov-11} can written
as follows
\[
\re\nint \overline{-i\psi_\kk} \cdot \psi_t \sp \kk = 
\no{\psi_t \sp\kk} \no{\psi_\kk}.
\]
Thus, we have an equality between the scalar product and the product
of norms. Then, there are $ \tilde{\om}_\kk\in\R $ such that
\begin{equation}
\label{eq:thm:lyapunov-18}
\psi_t \sp \kk = -i\tilde{\om}_\kk \psi_\kk.
\end{equation}
Taking the limit in the first equality of \eqref{eq:thm:lyapunov-3},
we obtain
\[
\om_\kk = \frac{\Ce_\kk}{\no{\psi_\kk}\sp 2}.
\]
From \eqref{eq:thm:lyapunov-18} and \eqref{eq:thm:lyapunov-11}, 
we also obtain $ \tilde{\om}_\kk = \om_\kk $. Hence
\[
\psi_\kk \sp t = -i\om_\kk \psi_\kk.
\]
By the second limit in \eqref{eq:thm:lyapunov-13} and the first
of \eqref{eq:thm:lyapunov-12}
\begin{equation}
\label{eq:thm:lyapunov-17}
\psi_n \sp{\kk,t}\ra\psi_t \sp\kk = -i\om_\kk \lambda_\kk u_\kk
\text{ in } L\sp 2 (\RN,\C).
\end{equation}
Thus, \eqref{eq:thm:lyapunov-16} and \eqref{eq:thm:lyapunov-17} give
\[
d((\psi_n,\psi_n \sp t),\Gamma_\Ce)\ra 0
\]
whence $ d((\phi_n,\phi_n \sp t),\Gamma_\Ce)\ra 0 $.
\end{proof}
\def\proofname{Proof}
The theorem can be restated by saying that 
\begin{gather*}
\VH\colon\Y\ra\R\\
(\phi,\phi_t)\mapsto (\EH(\phi,\phi_t) - m_\Ce)\sp 2 + 
\sum_{\kk = 1} \sp 2 (\CH_\kk (\phi,\phi_t) - \Ce_\kk)\sp 2
\end{gather*}
is a Lyapunov function for $ \Gamma_\Ce $, that is, 
$ \dist(\Phi_n,\Gamma_\Ce)\ra 0 $
if and only $ \VH(\Phi_n)\ra 0 $.
\nocite{Lio84-I,SS85,GSS87,GSS90,Zha03,BF10,BF10-2}
\bibliographystyle{siam}
\def\cprime{$'$} \def\cprime{$'$} \def\cprime{$'$} \def\cprime{$'$}
  \def\cprime{$'$} \def\cprime{$'$} \def\cprime{$'$} \def\cprime{$'$}
  \def\cprime{$'$} \def\polhk#1{\setbox0=\hbox{#1}{\ooalign{\hidewidth
  \lower1.5ex\hbox{`}\hidewidth\crcr\unhbox0}}}

\end{document}